\pdfoutput=1

\RequirePackage{plautopatch}
\documentclass{article}

\usepackage{adjustbox}
\usepackage{amsmath,amssymb}
\usepackage{amsthm} 
\usepackage{ascmac}
\usepackage{autobreak}
\usepackage{bbm}
\usepackage{bm} 
\usepackage{braket}
\usepackage{cite}
\usepackage{dcolumn} 
\usepackage{enumitem}
\usepackage{footnote}
\usepackage{framed}
\usepackage{graphicx}
\usepackage{here}
\usepackage{hyperref}
\usepackage{ifthen}
\usepackage{longtable}
\usepackage{mathtools}
\usepackage{multirow}
\usepackage[square,numbers,sort&compress]{natbib}
\usepackage{pict2e}
\newboolean{arxiv}
\setboolean{arxiv}{true}
\ifthenelse{\boolean{arxiv}}{}{\usepackage{pxjahyper}}
\usepackage{pxrubrica}
\usepackage[usestackEOL]{stackengine}
\usepackage{txfonts} 
\usepackage{xparse}
\usepackage[only,llparenthesis,rrparenthesis,mapsfrom]{stmaryrd}

\usepackage{color}
\usepackage[most]{tcolorbox}
\usepackage{tikz}
\usetikzlibrary{arrows.meta}
\usetikzlibrary{calc}
\usetikzlibrary{cd}

\usepackage{settings_common}
\usepackage{settings}

\usepackage{settings_common}
\usepackage{settings}

\usepackage[top=25truemm,bottom=20truemm,left=40truemm,right=40truemm]{geometry}

\newcommand{\titlename}{Characterization of vector spaces by isomorphisms}

\newcommand{\hypercolor}{black}
\hypersetup{
  colorlinks,
  citecolor=\hypercolor,
  linkcolor=\hypercolor,
  urlcolor=\hypercolor,
  bookmarksopen,
  bookmarksopenlevel=4,
  bookmarksnumbered,
}

\begin{document}

\title{\titlename}
\author{Kenji Nakahira \thanks{\href{mailto:nakahira@lab.tamagawa.ac.jp}{nakahira@lab.tamagawa.ac.jp}}}
\date{}

\maketitle

\begin{abstract}
 A vector space is commonly defined as a set that satisfies several conditions
 related to addition and scalar multiplication.
 However, for beginners, it may be hard to immediately grasp the essence of these conditions.
 There are probably a fair number of people who have wondered if these conditions could be
 substituted with ones that seem more straightforward.
 This paper presents a simple characterization of a finite-dimensional vector space,
 using the concept of an isomorphism, aimed at readers with a fundamental understanding of linear algebra.
 An intuitive way to understand an $N$-dimensional vector space would be to perceive it
 as a set (equipped with addition and scalar multiplication) that is isomorphic to
 the set of all column vectors with $N$ components.
 The method proposed in this paper formalizes this intuitive understanding in a straightforward manner.
 This method is also readily extendable to infinite-dimensional vector spaces.
 While this perspective may seem trivial to those familiar with algebra,
 it may be useful for those who have just started learning linear algebra and are
 contemplating the above question.
 Moreover, this approach can be generalized to free semimodules over a semiring.
\end{abstract}

\setcounter{tocdepth}{1}
\tableofcontents

\section{Characterization of finite-dimensional vector spaces} \label{sec:finite}

In this section, after making some preparations, we will present a definition of
a finite-dimensional vector space, the equivalence of which to the standard definition
will be demonstrated in the next section.

Let us arbitrarily choose a field $\Field$.
We will refer to the elements of $\Field$ as scalars.
The reader not familiar with fields may think of $\Field$ as being the set of all real numbers $\Real$.

\begin{define}{Set equipped with addition and scalar multiplication}{Set}
 Consider a set $X$.
 When an element, $x + y$, of $X$, called the \termdef{addition of $x$ and $y$}, is determined
 for each $x, y \in X$
 [or equivalently, a map $X \times X \ni (x,y) \mapsto x + y \in X$ exists],
 $X$ is said to be \termdef{equipped with an addition}.
 Also, when an element, $a \cdot x$ (simply denoted by $ax$), of $X$,
 called the \termdef{$a$ times of $x$},
 is determined for each $a \in \Field$ and $x \in X$
 [or equivalently, a map $\Field \times X \ni (a,x) \mapsto ax \in X$ exists],
 $X$ is said to be \termdef{equipped with scalar multiplication}.
\end{define}
 
\begin{define}{Linear map}{Linear}
 Consider two sets, $X$ and $Y$, equipped with addition and scalar multiplication.
 A map $f \colon X \to Y$ is called \termdef{linear} (or \termdef{homomorphism})
 if the following two conditions:
 \begin{enumerate}[label=(\alph*)]
  \item Preservation of addition: $f(x + y) = f(x) + f(y)$ holds for any $x, y \in X$.
  \item Preservation of scalar multiplication: $f(ax) = a \cdot f(x)$ holds
        for any $a \in \Field$ and $x \in X$.
 \end{enumerate}
\end{define}
Note that in these equations, $x + y$ and $ax$ on the left-hand sides are, respectively,
the sum and scalar multiplication defined in $X$,
while $f(x) + f(y)$ and $a \cdot f(x)$ on the right-hand sides are those defined in $Y$.

\begin{define}{Isomorphism}{Isomorphism}
 A linear map is called an \termdef{isomorphism} if it is invertible%
 \footnote{For a map $f \colon X \to Y$, if there exists a map $g \colon Y \to X$ such that
 both the compositions $g \c f$ and $f \c g$ are identity maps, then $g$, denoted by $f^{-1}$,
 is called the inverse of $f$ and $f$ is said to be \termdef{invertible}.}.
 For two sets $X$ and $Y$ equipped with addition and scalar multiplication,
 if there exists an isomorphism from $X$ to $Y$, then $X$ and $Y$ are said to be \termdef{isomorphic},
 denoted as $X \cong Y$.
\end{define}
\begin{supplemental}
 While these definitions differ from the standard definitions in linear algebra,
 they are equivalent to the standard definitions when both $X$ and $Y$ are vector spaces.
\end{supplemental}

Let us consider the $N$-dimensional column vector space $\Field^N$ for a non-negative integer $N$,
which is defined by
\begin{alignat}{1}
 \Field^N \coloneqq
 \begin{dcases}
  \{ s \coloneqq [ s_1,s_2,\dots,s_N ]^\T \mid s_1,s_2,\dots,s_N \in \Field \}, & N > 0, \\
  [0], & N = 0, \\
 \end{dcases}
 \label{eq:FN}
\end{alignat}
where $^\T$ denotes transpose.
$\Field^N$ is equipped with the following defined addition and scalar multiplication:
\begin{itemize}
 \item Addition: $s + t \coloneqq [s_1 + t_1, s_2 + t_2, \dots, s_N + t_N]^\T$
       for each $s,t \in \Field^N$.
 \item Scholar multiplication: $as \coloneqq [as_1, as_2, \dots, as_N]^\T$
       for each $a \in \Field$ and $s \in \Field^N$.
\end{itemize}

In this paper, we define a finite-dimensional vector space as follows.
\begin{define}{$N$-dimensional vector space}{VectorSpaceFinite}
 A set, $\V$, equipped with addition and scalar multiplication is called
 an \termdef{$N$-dimensional vector space}
 if there exists a non-negative integer $N$ such that $\V \cong \Field^N$.
\end{define}

Intuitively speaking, an $N$-dimensional vector space is defined as a set that follows
the same operation rules for addition and scalar multiplication as $\Field^N$.
While this definition presupposes a specific set, $\Field^N$,
it does not explicitly state several conditions that are present in the standard definition.
This definition seems to be more straightforward in the sense that it is based on a linear map,
a concept often regarded as central in linear algebra.
\begin{supplemental}
 From one perspective, several conditions in the standard definition could be seen as
 explicitly detailing the requirements for having the same operational rules as $\Field^N$
 with respect to addition and scalar multiplication.
 The above definition could be interpreted as a more direct expression of having
 the same operational rules as $\Field^N$.
\end{supplemental}

\section{Equivalence with the standard definition}

In this section, we demonstrate our definition is equivalent to the standard definition
of a finite-dimensional vector space.
To distinguish these two difinitions, we will refer to a vector space defined in the
following standard way as the standard notion of a vector space.
\begin{define}{Standard notion of vector space}{StandardVectorSpace}
 A set, $\V$, equipped with addition and scalar multiplication is called the
 \termdef{standard notion of a vector space} if $\V$ contains an element $\zero$
 and satisfies the following conditions for each $a,b \in \Field$ and $x,y,z \in \V$:

 \vspace{.5em}
 \begin{tabular}[tb]{ll}
  (1)~$x + \zero = x$,
  & (2)~$x + (y + z) = (x + y) + z$, \\
  (3)~$x + y = y + x$,
  & (4)~$1 x = x$, \\
  (5)~$a(b x) = (ab)x$,
  & (6)~$a(x + y) = ax + ay$, \\
  (7)~$(a + b)x = ax + bx$,
  & (8)~$0 x = \zero = a \zero$. \\
 \end{tabular}
\end{define}
\begin{supplemental}[1]
 Here, we took into account Conditions (1)--(8) with the intention of later extending this definition.
 Instead of Condition~(8), the following condition is often assumed \cite{Shi-1977}:
 
 \vspace{.5em}
 \begin{tabular}[tb]{ll}
  \hspace{2em} (9)~$-x \in \V$ exists such that $x + (-x) = \zero$.
 \end{tabular}
 \vspace{.5em}

 \envindent
 Note that when Conditions (2)--(7) hold, $0x = \zero$ of Condition~(8) is equivalent to
 Conditions~(1) and (9)%
 \footnote{Proof: If $0x = \zero$ holds, then from $x + \zero = x + 0x = 1x + 0x = 1x = x$,
 Condition~(1) holds, and from $x + (-1)x = 1x + (-1)x = 0x = \zero$, Condition~(9) holds.
 Conversely, if Conditions (1) and (9) hold, then
 $0x = 0x + \zero = 0x + 0x + (-0x) = 0x + (-0x) = \zero$ holds.},
 and if these conditions hold, then $a \zero = \zero$ holds%
 \footnote{Proof: $a \zero = a \zero + \zero = a \zero + a \zero + (-a \zero)
 = a(\zero + \zero) + (-a \zero) = a \zero + (-a \zero) = \zero$.}.
 Therefore, Condition~(1) and $a \zero = \zero$ of Condition~(8) is redundant.
\end{supplemental}
\begin{supplemental}[2]
 Other sets of conditions can be considered.
 For example, we can omit Condition~(3) at the cost of replacing Condition (1) with
 $x + \zero = x = \zero + x$%
 \footnote{Indeed, $x + x + y + y = (1 + 1)x + (1 + 1)y = (1 + 1)(x + y)
 = x + y + x + y$ holds, and thus adding $(-1)x$ from the left and $(-1)y$ from the right
 to both sides yields Condition~(3).}.
 Thus, in the standard notion, we can say that there is a certain degree of flexibility
 in choosing the set of conditions.
\end{supplemental}

We call $\V$ the \termdef{standard notion of an $N$-dimensional vector space}
if there exists a set, $Y$, of $N$ elements, called a \termdef{basis} of $\V$, such that
each $x \in \V$ can be uniquely expressed in the form of
$x = \sum_{y \in Y} x_y \cdot y$ with $x_y \in \Field$.

In what follows, we prove the following proposition.
\begin{proposition}{}{main}
 A set equipped with addition and scalar multiplication is an $N$-dimensional vector space
 if and only if it is the standard notion of an $N$-dimensional vector space.
\end{proposition}
\begin{proof}
 First, we show that if $\V$ is the standard notion of an $N$-dimensional vector space, then
 it is an $N$-dimensional vector space.
 Let us choose a basis $\{ u_n \}_{n=1}^N$ of $\V$; then,
 each element of $\V$ can be uniquely expressed in the form
 $\sum_{n=1}^N a_n u_n$ with $a_n \in \Field$.
 It is easily seen that the map
 $\V \ni \sum_{n=1}^N a_n u_n \mapsto [a_1,a_2,\dots,a_N]^\T \in \Field^N$
 is an isomorphism, whose inverse is
 $\Field^N \ni [a_1,a_2,\dots,a_N]^\T \mapsto \sum_{n=1}^N a_n u_n \in \V$.
 Therefore, $\V$ is an $N$-dimensional vector space.
 
 Next, we show that if $\V$ is an $N$-dimensional vector space, then it is the standard notion
 of an $N$-dimensional vector space.
 Let $\zero' \coloneqq [ 0, 0, \dots, 0]^\T \in \Field^N$.
 By the definition of a vector space, there exists an isomorphism $\psi \colon \V \to \Field^N$.
 For any $x,y,z \in \V$, $\psi(x), ~\psi(y), ~\psi(z) \in \Field^N$ satisfy the following conditions,
 which correspond to Conditions~(1)--(8), for any $a,b \in \Field$.

 \vspace{.5em}
 \begin{tabular}[tb]{ll}
  (1')~$\psi(x) + \zero' = \psi(x)$,
  & (2')~$\psi(x) + [\psi(y) + \psi(z)] = [\psi(x) + \psi(y)] + \psi(z)$, \\
  (3')~$\psi(x) + \psi(y) = \psi(y) + \psi(x)$,
  & (4')~$1 \cdot \psi(x) = \psi(x)$, \\
  (5')~$a[b \cdot \psi(x)] = (ab)\psi(x)$,
  & (6')~$a[\psi(x) + \psi(y)] = a \cdot \psi(x) + a \cdot \psi(y)$, \\
  (7')~$(a + b) \cdot \psi(x) = a \cdot \psi(x) + b \cdot \psi(x)$,
  & (8')~$0 \cdot \psi(x) = \zero' = a \zero'$. \\
 \end{tabular}
 \vspace{.5em}

 \noindent
 Expressing both sides of each of these equations in the form of $\psi(\cdots)$ and
 then applying $\psi^{-1}$ yields Conditions~(1)--(8).
 For example, the left-hand side of Condition~(1') is expressed by $\psi(x) + \psi(\zero) = \psi(x + \zero)$,
 where $\zero \coloneqq \psi^{-1}(\zero') \in \V$, so applying $\psi^{-1}$ gives $x + \zero$.
 Also, applying $\psi^{-1}$ to the right-hand side gives $x$, so we have $x + \zero = x$,
 i.e., Condition~(1).
 Similarly, the left-hand side of Condition~(2') is $\psi(x) + \psi(y + z) = \psi[x + (y + z)]$,
 and the right-hand side is $\psi(x + y) + \psi(z) = \psi[(x + y) + z]$, so applying $\psi^{-1}$
 to both sides yields $x + (y + z) = (x + y) + z$, i.e., Condition~(2).
 Conditions (3)--(8) can be obtained in the same way.
 Therefore, $\V$ is the standard notion of an $N$-dimensional vector space.
\end{proof}

As can be seen from this proof, if we define a vector space as in this paper, then
Conditions (1)--(8) of the standard definition are derived by the operation rules for addition
and scalar multiplication
in $\Field^N$, i.e., Conditions (1')--(8'), via the isomorphism $\psi$ and its inverse.

It is worth mentioning about our definition.
In general, an operation called addition is required to satisfy some basic conditions
such as Conditions~(1)--(3) of Definition~\ref{def:StandardVectorSpace}.
In contrast, in Definition~\ref{def:Set}, no conditions are imposed on the addition.
The same applies to scalar multiplication, and only the equipment of these two operations is required.
Proposition~\ref{pro:main} can be interpreted as claiming that just by requiring a set
equipped with addition and scalar multiplication to be isomorphic to $\Field^N$
as in Definition~\ref{def:VectorSpaceFinite}, Conditions (1)--(8) inherently hold.

\section{Extension to infinite-dimensional vector spaces and free semigroups over semirings}

In this section, we discuss how the ideas presented in the previous sections can be extended
to infinite-dimensional vector spaces and free semigroups over semirings.

\subsection{Extension to infinite-dimensional vector spaces}

It is easy to extend the above ideas to infinite-dimensional vector spaces.
\begin{define}{$\Field^J$}{FJ}
 For a set $J$, we denote by $\Field^J$ (also written as $\bigoplus_J \Field$)
 the set of maps $s \colon J \ni j \mapsto s_j \in \Field$
 such that the number of $j \in J$ satisfying $s_j \neq 0$ is finite,
 where $\Field^J$ is assumed to have addition and scalar multiplication defined as follows:
 \begin{itemize}
  \item Addition: for each $s,t \in \Field^J$, let $s + t$ be the map
        $J \ni j \mapsto s_j + t_j \in \Field$.
        Note that $s + t$ is in $\Field^J$.
  \item Scholar multiplication: for each $a \in \Field$ and $s \in \Field^J$,
        let $as$ be the map $J \ni j \mapsto a \cdot s_j \in \Field$.
        Note that $as$ is in $\Field^J$.
 \end{itemize}
\end{define}
\begin{define}{Vector space}{VectorSpace}
 A set, $\V$, equipped with addition and scalar multiplication is called
 a \termdef{vector space} if there exists a set $J$ such that $\V \cong \Field^J$.
\end{define}

\begin{ex}{}{}
 For a non-negative integer $N$, let $\tilde{N} \coloneqq \{ 1,2,\dots,N \}$; then,
 $\Field^N \cong \Field^{\tilde{N}}$ holds.
 Indeed, the map that maps each $[s_1,s_2,\dots,s_N]^\T \in \Field^N$ to
 the map $\tilde{N} \ni n \mapsto s_n \in \Field$ is an isomorphism from $\Field^N$
 to $\Field^{\tilde{N}}$.
 Thus, any $N$-dimensional vector space defined in Definition~\ref{def:VectorSpaceFinite}
 is a vector space.
\end{ex}

\begin{proposition}{}{main2}
 A set equipped with addition and scalar multiplication is a vector space
 if and only if it is the standard notion of a vector space.
\end{proposition}
\begin{proof}
 If $\V$ is a vector space, it can be shown to be the standard notion of a vector space
 in the same way as the proof of Proposition~\ref{pro:main}.
 The converse can be shown as follows.
 Arbitrarily choose a basis $Y$ of $\V$%
 \footnote{A subset, $Y$, of $\V$ is called a \termdef{basis} of $\V$
 if each $x \in \V$ can be uniquely expressed in the form of
 $x = \sum_{y \in Y} x_y \cdot y$, where each $x_y$ is in $\Field$
 and the number of $y \in Y$ satisfying $x_y \neq 0$ is finite.
 Under the axiom of choice, it is known that every vector space has a basis.}.
 Each $x \in \V$ can be uniquely expressed in the form
 $x = \sum_{y \in Y} x_y \cdot y \in \V$ with $x_y \in \Field$, and we consider a map $\psi$
 from $\V$ to $\Field^Y$ that maps $x$ to the map $Y \ni y \mapsto x_y \in \Field$.
 It is immediately seen that $\psi$ is an isomorphism and its inverse $\psi^{-1}$ is
 $\Field^Y \ni s \mapsto \sum_{y \in Y} s_y \cdot y \in \V$.
 Therefore, since $\V \cong \Field^Y$ holds, $\V$ is a vector space.
\end{proof}

\subsection{Extension to free semigroups over semirings}

The above discussion can be extended in the same way to free semimodules over a semiring%
\footnote{We call $R$ a \termdef{semiring} if it contains two elements $0$ and $1$
(which may coincide) and the following conditions hold for each $a,b,c \in R$:

\vspace{.5em}
\begin{tabular}[tb]{ll}
 (1)~$a + 0 = a$,
 & (2)~$a + (b + c) = (a + b) + c$, \\
 (3)~$a + b = b + a$,
 & (4)~$1 a = a = a 1$, \\
 (5)~$a(b c) = (ab)c$,
 & (6)~$a(b + c) = ab + ac$, \\
 (7)~$(a + b)c = ac + bc$,
 & (8)~$0 a = 0 = a 0$. \\
\end{tabular}
\vspace{.5em}

Furthermore, $R$ is called a \termdef{ring} (with unity) if, for each $a \in R$, $-a \in R$
exists such that $a + (-a) = 0$.}%
$R$.
Considering the simplicity of the extension, we will limit our discussion to the main points.

Let us arbitrarily choose a semiring $R$ and refer to its elements as scalars.
By replacing $\Field$ with $R$ in Definitions~\ref{def:Set}, \ref{def:Linear}, and \ref{def:Isomorphism},
we can define sets equipped with addition and scalar multiplication, linear maps, and isomorphisms.

The standard notion of a free $R$-semimodule is defined as follows:
\begin{define}{Standard notion of free $R$-semimodule}{}
 Let $R$ be a semiring with elements referred to as scalars.
 A set, $\M$, equipped with addition and scalar multiplication is called the
 \termdef{standard notion of an $R$-semimodule} (or a left $R$-semimodule) if $\M$ contains
 an element $\zero$ and satisfies the following conditions for each $a,b \in R$
 and $x,y,z \in \M$:

 \vspace{.5em}
 \begin{tabular}[tb]{ll}
  (1)~$x + \zero = x$,
  & (2)~$x + (y + z) = (x + y) + z$, \\
  (3)~$x + y = y + x$,
  & (4)~$1 x = x$, \\
  (5)~$a(b x) = (ab)x$,
  & (6)~$a(x + y) = ax + ay$, \\
  (7)~$(a + b)x = ax + bx$,
  & (8)~$0 x = \zero = a \zero$. \\
 \end{tabular}
 \vspace{.5em}

 Moreover, $\M$ is called the \termdef{standard notion of a free $R$-semimodule}
 if there exists a subset, $Y$, of $\M$ such that each $x \in \M$ can be uniquely expressed
 in the form $x = \sum_{y \in Y} x_y \cdot y$, where each $x_y$ is in $R$
 and the number of $y \in Y$ satisfying $x_y \neq 0$ is finite.
\end{define}

Note that a free $R$-semimodule is equivalent to a free $R$-module if $R$ is a ring,
and to a vector space over $R$ if $R$ is a field.
\begin{ex}{}{}
 Let $N$ be a positive integer
 and $\Realp$ be the set of all non-negative real numbers, which is a semiring.
 Consider the set
 \begin{alignat}{1}
  \Realp^N \coloneqq \{ [ r_1,r_2,\dots,r_N ]^\T \mid r_1,r_2,\dots,r_N \in \Realp \},
 \end{alignat}
 which replaces $\Field$ in the set $\Field^N$ of Eq.~\eqref{eq:FN} with $\Realp$,
 where addition and scalar multiplication are defined as component-wise
 addition and scalar multiplication in the same manner as $\Field^N$.
 $\Realp^N$ is the standard notion of a free $\Realp$-semimodule.
 Note that since $\Realp$ is not a ring, $\Realp^N$ is not a module over a ring.
\end{ex}
\begin{ex}{}{}
 Let $N$ be a positive integer and
 $\Integer$ be the set of all integers, which is a ring.
 Consider the set
 \begin{alignat}{1}
  \Integer^N \coloneqq \{ [ m_1,m_2,\dots,m_N ]^\T \mid m_1,m_2,\dots,m_N \in \Integer \},
 \end{alignat}
 which replaces $\Field$ in the set $\Field^N$ of Eq.~\eqref{eq:FN} with $\Integer$,
 where addition and scalar multiplication are defined as component-wise
 addition and scalar multiplication in the same manner as $\Field^N$.
 $\Integer^N$ is the standard notion of a free $\Integer$-module.
 Note that since $\Integer$ is not a field, $\Integer^N$ is not a vector space.
\end{ex}

As a generalization of $\Field^J$ in Definition~\ref{def:FJ}, we can define $R^J$ as follows:
\begin{define}{$R^J$}{}
 Let $R$ be a semiring with elements referred to as scalars.
 For a set $J$, we denote by $R^J$ the set of maps $s \colon J \ni j \mapsto s_j \in R$
 such that the number of $j \in J$ satisfying $s_j \neq 0$ is finite,
 where $R^J$ is assumed to have addition and scalar multiplication defined as follows:
 \begin{itemize}
  \item Addition: for each $s,t \in R^J$, let $s + t$ be the map
        $J \ni j \mapsto s_j + t_j \in R$.
        Note that $s + t$ is in $R^J$.
  \item Scholar multiplication: for each $a \in R$ and $s \in R^J$,
        let $as$ be the map $J \ni j \mapsto a \cdot s_j \in R$.
        Note that $as$ is in $R^J$.
 \end{itemize}
\end{define}
A free $R$-semimodule can be defined as follows:
\begin{define}{Free $R$-semimodule}{Semimodule}
 Let $R$ be a semiring with elements referred to as scalars.
 A set, $\M$, equipped with addition and scalar multiplication is called
 a \termdef{free $R$-semimodule} if there exists a set $J$ such that $\V \cong R^J$.
\end{define}

This definition is equivalent to the standard definition.

\begin{proposition}{}{}
 Let $R$ be a semiring.
 A set equipped with addition and scalar multiplication is a free $R$-semimodule
 if and only if it is the standard notion of a free $R$-semimodule.
\end{proposition}
The proof follows the same method as the proof of Proposition~\ref{pro:main2} and thus is omitted.
\begin{supplemental}
 The proposed approach can be applied even in more general cases than free $R$-semimodules
 (for example, when $R$ is not necessarily a semiring),
 or when equipped with operations other than addition and scalar multiplication.
\end{supplemental}

\section*{Acknowledgment}

I am grateful to K. Kato for insightful discussions.

\end{document}